\documentclass[10pt,a4paper]{article}

\usepackage[T1]{fontenc} 
\usepackage[utf8]{inputenc} 
\usepackage{graphicx}
\usepackage[all,line]{xy}

\usepackage[english]{babel} 
\usepackage{natbib}
\usepackage{mdwlist}
\usepackage{enumerate}
\usepackage[draft]{fixme}

\usepackage{amsmath,amssymb,stmaryrd,mathtools,dsfont}
\newcommand\set{\mathbb}
\newcommand\cat{\mathcal}
\newcommand\id{\mathds{1}}
\newcommand\ideal{\mathfrak}
\newcommand\Lie{\mathfrak}
\newcommand\Kac[1]{\tilde{\Lie #1 }}
\newcommand\Ker{\text{Ker}}
\newcommand\Cok{\text{Coker}}
\newcommand\Hom{\text{Hom}}

\newcommand\Ext{\text{Ext}}

\newcommand\Gr{\text{K}_0}
\newcommand\tensor{\otimes}

\newcommand\rank{\text{rank}}

\newcommand\codim{\text{codim}}
\newcommand\pd{\text{proj dim}}
\renewcommand\epsilon{\varepsilon}
\renewcommand\phi{\varphi}
\renewcommand\tilde{\widetilde}

\newcommand\inner[2]{\left( #1 , #2 \right)}
\newcommand\cartan[2]{\left\langle #1 , #2^\vee \right\rangle}
\newcommand\defineset[2]{\lbrace #1 \ \vert\ #2 \rbrace}

\newcommand\indexingset{\mathcal I}

\usepackage[amsmath,thmmarks]{ntheorem}
\theoremseparator{.}
\newtheorem{theorem}{Theorem}
\newtheorem{lemma}[theorem]{Lemma}

\newtheorem{definition}[theorem]{Definition}
\theoremstyle{nonumberplain}
\theorembodyfont{\normalfont}
\theoremheaderfont{\itshape}
\newtheorem{remark}{Remark}
\theoremsymbol{\enspace\ensuremath{\square}}
\newtheorem{proof}{Proof}

\date{\today}

\title{Rank 2 fusion rings are complete intersections}
\author{Troels Bak Andersen}

\begin{document}

\maketitle

\begin{abstract} 
  We give a non-constructive proof that fusion rings attached to a simple complex Lie algebra of rank 2 are complete intersections. 
\end{abstract}

\section{Introduction}

Attached to a simple complex Lie algebra $\Lie g$ and a natural number $k \geq 0$ we have the level $k$ fusion ring $F_k$. As a $\set Z$-module it is free and finitely generated over the dominant weights of $\Lie g$ of level at most $k$ and has a product structure making it a commutative, associative, unital ring. It can be described as a quotient of the representation ring $\mathcal R$ by the ``fusion ideal'', which has an infinite generating set enumerated by all dominant weights of level strictly higher than $k$. We consider the problem of finding a minimal generating set of this ideal. 

If $\rank(\Lie g) = r$ then as a commutative, associative $\set Z$-algebra $\mathcal R$ is isomorphic to the polynomial ring $R = \set Z[X_1, \dots, X_r]$ and $F_k$ to the quotient $R/I_k$, where $I_k \subseteq R$ corresponds to the fusion ideal. Then, as $\codim(I_k,R) = r$, the best we can hope for is a generating set of $I_k$ consisting of $r$ elements, which in turn must constitute an $R$-regular sequence meaning $R/I_k$ is a complete intersection ring. 

This goal was achieved for $\Lie g$ of type $A_r, r \geq 1$ by \cite{Gep91} and type $C_r, r \geq 2$ by \cite{BMRS92} by finding an explicit potential function $V_{k+h^\vee}(X_1, \dots, X_r) \in R$ for each level $k$ whose $r$ partial derivatives generate the fusion ideal. For Lie algebras of the remaining types we will have to use a different approach as it was proven by \cite{BR06} that in these cases the fusion ideal cannot be described by analogous potential functions. 


In this paper we prove non-constructively that the fusion ideal for a rank 2 fusion ring can always be generated by 2 elements. This settles in affirmative the question above for the remaining rank 2 case of type $G_2$ though it doesn't give explicit generators. 

The author owes thanks to Professor Shrawan Kumar, who introduced him to the problem during his visit at UNC Chapel Hill in the fall 2012, and also helped him with the first outline of a proof. He would also like to thank Professor Mohan Kumar at Washington University in St. Louis and Professor Holger Andreas Nielsen at Aarhus University for great help at getting the last details right. 

The author was supported by the Danish National Research Foundation center of Excellence, Center  for  Quantum  Geometry  of  Moduli  Spaces  (QGM).

\section{Definition and some properties of fusion rings}

Consider an abelian category $\cat F$, which is $\set C$-linear and semisimple, i.e., the $\Hom$-spaces are finite-dimensional $\set C$-vector spaces, composition of morphisms is $\set C$-linear and there is a countable collection of simple objects $\defineset{V_\lambda}{\lambda \in \indexingset}$, such that $\dim_{\set C} \Hom_{\cat F}(V_\lambda, V_\mu) = \delta_{\lambda,\mu}$, and any object in $\cat F$ is isomorphic to a finite direct sum of the $V_\lambda$. Assume furthermore that the category is rigid, monoidal and that the unit object is simple. The Grothendieck group $F = \Gr(\cat F)$ is called a fusion ring, and if there are only finitely many isomorphism classes of simple objects, then $F$ is called a rational fusion ring. 

Examples of settings associated to a semisimple Lie algebra $\Lie g$ in which fusion rings arise includes: 
\begin{enumerate}[{Case} 1]
\item\label{exone} The category of finite-dimensional $\Lie g$-modules.
\item\label{extwo} The category of fixed-level representations of the affine Kac-Moody algebra $\Kac g$ associated to $\Lie g$, see e.g.\ \cite{Bea96}.
\item\label{exthree} The category of tilting modules of the associated quantum group at a complex root of unity, see e.g.\ \cite{And92}.
\item\label{exfour} The category of rational modules of the corresponding semisimple, simply connected algebraic group over a field of positive characteristic, see e.g. \cite{AP95}. 
\end{enumerate}
Of these four examples the last three lead to rational fusion rings. In section \ref{setup} we will review the Cases \ref{exone} and \ref{extwo}.

Let $\defineset{[\lambda]}{\lambda \in \indexingset}$ denote a $\set Z$-basis for $F$ corresponding to the isomorphism classes of the simple objects in $\cat F$. The ring structure on $F$ relative to this basis is given by
\begin{equation*}
  [\lambda][\mu] = \sum_{\nu \in \indexingset} N_{\lambda,\mu}^\nu [\nu],
\end{equation*}
where $N_{\lambda,\mu}^\nu = \dim_{\set C} \Hom_{\cat F}(V_\lambda \tensor V_\mu, V_\nu) \geq 0$ is the multiplicity with which $V_\nu$ occurs in a decomposition of $V_\lambda \tensor V_\mu$ into simple summands. The duality functor $V \mapsto V^*$ on $\cat F$ has $V \simeq (V^*)^*$, so it maps simple objects to simple objects, giving an involution $\lambda \mapsto \lambda^*$ of $\indexingset$, which induces an antiautomorphism of $F$. If we let $\lambda_0 \in \indexingset$ correspond to the unit object in $\cat F$, this means that $N_{\lambda, \mu}^{\lambda_0} = \delta_{\lambda, \mu^*}$ together with $N_{\lambda^*, \mu^*}^{\nu^*} = N_{\lambda,\mu}^\nu$. We end up with an axiomatization of a fusion ring independent of the category theory:

\begin{definition}
Let $\indexingset$ be a countable set. A fusion rule on $\indexingset$ is a set of non-negative integers $N = \defineset{N_{\lambda, \mu}^\nu \in \set N}{\lambda, \mu, \nu \in \indexingset}$ such that
  \begin{enumerate}[F1]
  \item \label{fusionRuleComm} $N_{\lambda, \mu}^\nu = N_{\mu, \lambda}^\nu$ for all $\lambda, \mu, \nu \in \indexingset$,
  \item \label{fusionRuleAss} $\sum_{\eta \in \indexingset} N_{\lambda, \mu}^\eta N_{\eta, \nu}^\zeta = \sum_{\eta \in \indexingset} N_{\mu, \nu}^\eta N_{\lambda, \eta}^\zeta$ for all $\lambda, \mu, \nu, \zeta \in \indexingset$,
  \item \label{fusionRuleConj} there is an element $\lambda_0 \in \indexingset$ and an associated map $\lambda \mapsto \lambda^*$ given by $N_{\lambda, \mu}^{\lambda_0} = \delta_{\mu, \lambda^*}$ such that $(\lambda^*)^* = \lambda$ and $N_{\lambda^*, \mu^*}^{\nu^*} = N_{\lambda, \mu}^\nu$ for all $\lambda, \mu, \nu \in \indexingset$. 
  \end{enumerate}
  The associated fusion ring $F = F(N)$ is the free $\set Z$-module with basis $\defineset{[\lambda]}{\lambda \in\nobreak \indexingset}$ equipped with multiplication $[\lambda][\mu] = \sum_{\nu \in \indexingset} N_{\lambda, \mu}^\nu [\nu]$. If the indexing set $\indexingset$ is finite the fusion ring has finite rank and we say it is rational. 
\end{definition}

The assumptions F\ref{fusionRuleComm}-\ref{fusionRuleConj} translates directly into the multiplication in $F$ being commutative and associative and the involution being a homomorphism. 

Given an $x \in F$ write it in the basis given by $\indexingset$ as $x = \sum_{\lambda \in \indexingset} n_\lambda [\lambda]$ and define a $\set Z$-linear form $t(x) = n_{\lambda_0}$. By assumption it satisfies $t([\lambda][\mu]^*) = N_{\lambda, \mu^*}^{\lambda_0} = \delta_{\lambda, \mu}$. For $x,y \in F$ define a $\set Z$-bilinear form $\inner{x}{y} = t(xy^*)$. It is positive-definite since $\inner{[\lambda]}{[\mu]} = \delta_{\lambda, \mu}$.

Assume from now on, that the indexing set $\indexingset$ is finite. Then the bilinear form defines an isomorphism of $\set Z$-modules of $F$ with $\Hom_{\set Z}(F, \set Z)$ by $x \mapsto [f_x: y \mapsto \inner{x}{y}]$. If we give $\Hom_{\set Z}(F, \set Z)$ an $F$-action by $zf(y) = f(z^*y)$, then this isomorphism is actually an isomorphism of $F$-modules since $f_{zx}(y) = \inner{zx}{y} = \inner{x}{z^*y} = zf_x(y)$ for all $z \in F$. 

The exact sequence $0 \to \set Z \to \set Q \to \set Q/\set Z \to 0$ induces to
\begin{equation*}
  0 \to \Hom_{\set Z}(F, \set Z) \to \Hom_{\set Z}(F, \set Q) \to \Hom_{\set Z}(F, \set Q/\set Z) \to 0.
\end{equation*}
Since $\set Q$ and $\set Q/\set Z$ are injective $\set Z$-modules $\Hom_{\set Z}(F, \set Q)$ and $\Hom_{\set Z}(F, \set Q/\set Z)$ are injective $F$-modules. This means that $\Hom_{\set Z}(F, \set Z)$ has finite injective dimension, and by the module isomorphism $F \simeq \Hom_{\set Z}(F, \set Z)$ constructed above, we conclude that so has $F$, i.e., rational fusion rings are Gorenstein. 

Let $R = \set Z[X_1, \dots, X_r]$ be a polynomial ring, that maps surjectively to $F$ with kernel $I$, and assume that $R$ has minimal rank. We consider finding the minimal number of generators of $I$ in $R$. Since $I$ has codimension $r$ in $R$, this is a lower bound on that number. The set of integers $\set Z$ is a 1-dimensional discrete valuation ring, so it is a regular ring. Then also $R$ is a regular ring. If we can prove existence of a generating set consisting of exactly $r$ elements, then they must constitute a regular $R$-sequence, i.e., $R/I$ is a complete intersection ring. In Theorem \ref{UNCproof} we prove, that this is always the case if $r=2$.

\section{The Lie algebra setup}\label{setup}

Consider a simple complex Lie algebra $\Lie g$ of rank $r$ with root system $\Phi$, Weyl group $W$ and weight lattice $P$. The irreducible finite-dimensional representations $V(\lambda)$ are characterized by having highest weights in $P_+ = \defineset{\lambda \in P}{0 \leq \cartan{\lambda}{\alpha}\text{ for all } \alpha \in \Phi_+}$. The representation ring $\mathcal R$ has as elements formal differences of isomorphism classes of finite-dimensional representations of $\Lie g$ with addition given by direct sums of representations and multiplication given by tensor product over $\set C$. We write $[\lambda]$ short for the isomorphism class of the irreducible representation with highest weight $\lambda = a_1\omega_1 + \dots + a_r\omega_r$, where the $\omega_i$ are the fundamental weights. The ring structure is encoded in the structure constants 
\begin{equation*}
M_{\lambda,\mu}^\nu = \dim_{\set C} \Hom_{\Lie g}(V(\lambda) \tensor V(\mu) \tensor V(\nu^*), \set C),
\end{equation*}
for which $V(\lambda) \tensor V(\mu) = \bigoplus_{\nu \in P_+} M_{\lambda,\mu}^\nu V(\nu)$, where $\nu^* = -w_0(\nu)$ is involution on $P$ given by the longest element $w_0 \in W$. 

This is an elaboration of Case \ref{exone}: $M = \defineset{M_{\lambda, \mu}^\nu}{\lambda, \mu, \nu \in P}$ defines a (non-rational) fusion rule on $P$ and its associated fusion ring is the representation ring. In the following we elaborate on Case \ref{extwo}, how to obtain a rational fusion ring as a quotient of $\mathcal R$. Notice for now, that $\mathcal R$ is freely generated by the isomorphism classes $X_i = [\omega_i]$ of the fundamental weights, i.e., it has a commutative presentation as the polynomial ring $R = \set Z[X_1, \dots, X_r]$. 

For a given level $k \geq 0$ we restrict to the alcove $P_k = \defineset{\lambda \in P_+}{\cartan{\lambda}{\beta_0} \leq k}$, where $\beta_0$ is the highest root in $\Phi$. Notice that the involution preserves the positive part of the root system, therefore it fixes the highest root, and consequently it restricts to an involution of $P_k$. 

Let $V_{\set P^1}^\dag(\lambda,\mu,\nu^*)$ denote the space of conformal blocks on $\set P^1$ with three distinct marked points and the weights $\lambda, \mu, \nu^* \in P_k$ attached to them with central charge $k + h^\vee$. We refer to \cite[2]{Bea96} or \cite[3.1]{BK09} for a rigorous definition of the spaces, but here we focus on the dimensions 
\begin{equation*}
  N_{\lambda, \mu}^\nu = \dim_{\set C} V_{\set P^1}^\dag(\lambda,\mu,\nu^*).
\end{equation*}
We will see in \eqref{fusionCoef} how these numbers can be calculated combinatorially. 

The claim is now that $N_k = \defineset{N_{\lambda, \mu}^\nu}{\lambda, \mu, \nu \in P_k}$ defines a rational fusion rule on $P_k$ with the associated fusion ring $F_k = F(N_k)$. That commutativity (F\ref{fusionRuleComm}) is satisfied is direct from the definition. That associativity (F\ref{fusionRuleAss}) is satisfied is a deep theorem, first proven as a result of the factorization rules in \cite{TUY89}. An analogue of \cite[2.8]{Bea96} shows that $V
_{\set P^1}^\dag(\lambda, \mu, \nu^*) \simeq V_{\set P^1}^\dag(\lambda^*, \mu^*, \nu)$ so $N_{\lambda, \mu}^\nu = N_{\lambda^*, \mu^*}^{\nu^*}$, which concludes that involution is a homomorphism (F\ref{fusionRuleConj}). Since $P_k$ is finite, $F_k$ is rational. By abuse of notation let also $[\lambda]$ denote a basis of $F_k$. 

Let $W_{k+h^\vee}$ denote the affine Weyl group generated by $W$ and the translation $\lambda \mapsto \lambda + (k + h^\vee)\beta_0$, where $h^\vee = \cartan{\rho}{\beta_0} + 1$ is the dual Coxeter number and $\rho = \omega_1 + \dots + \omega_r$. Then $P_k$ is a fundamental domain for the ($\rho$-shifted) $W_{k+h^\vee}$-action on $P$.

Now define a $\set Z$-linear map 
\begin{equation*}
  \pi: \mathcal R \to F_k
\end{equation*}
as follows. If a weight $\lambda \in P^+$ lies on an affine wall, i.e., $\lambda + \rho$ is fixed by an element of $W_{k+h^\vee}$, then $\pi([\lambda]) = 0$. Otherwise there is a unique $w \in W_{k+h^\vee}$ such that $w.\lambda = w(\lambda + \rho) - \rho \in P_k$, and we set $\pi([\lambda]) = \det(w)[w.\lambda]$. In particular $\pi([\lambda]) = [\lambda]$ for $\lambda \in P_k$, justifying the abuse of notation. 

It is a fact that $\pi$ is a ring homomorphism for all semisimple Lie algebras, cf. \cite[3.7]{BK09}. If $\lambda, \mu \in P_k$ this means that 
\begin{equation*}
  \sum_{\xi \in P_+} M_{\lambda, \mu}^\xi \pi([\xi]) = \sum_{\nu \in P_k} N_{\lambda, \mu}^\nu[\nu],
\end{equation*}
i.e.,
\begin{equation}\label{fusionCoef}
  N_{\lambda,\mu}^\nu = \sum_{w \in W_{k+h^\vee}}\det(w)M_{\lambda,\mu}^{w.\nu}.
\end{equation}

This identifies the fusion ring $F_k$ with the quotient of $\mathcal R$ by the kernel of the homomorphism. We consider the ideal $I_k \subseteq R$ corresponding to $\Ker(\pi) \subseteq \mathcal R$. The ideal is generated by the polynomials in $\set Z[X_1, \dots, X_r]$, which corresponds to the elements $[\lambda] - \det(w)[w.\lambda] \in \mathcal R$ with $\lambda, w.\lambda \in~P_+$ and $w \in W_{k+h^\vee}$. 

As mentioned in the introduction, when $\Lie g$ has type $A_r$ or $C_r$ it is known, that $I_k$ can be generated by $r$ elements, in fact in both cases the ideal is generated by the polynomials corresponding to $[(k+1)\omega_1]$ and $[k\omega_1 + \omega_2]$. 

For $r=2$ this only leaves the case $G_2$ not taken care of. The paper \cite{BK09} conjecturally described this fusion ideal as the radical of an ideal generated by three specific elements and later \cite{Dou13} found an actual generating set consisting of the three elements
\begin{equation*}
  \begin{cases}
    [\frac{k+2}{2}\omega_2] + [\frac{k}{2}\omega_2],\ [\omega_1 + \frac{k}{2}\omega_2]\text{ and }[3\omega_1 + \frac{k-2}{2}\omega_2] & \text{if }k\text{ even} \\
    [\frac{k+1}{2}\omega_2],\ [2\omega_1 + \frac{k-1}{2}\omega_2]\text{ and }[3\omega_1 + \frac{k-1}{2}\omega_2] + [3\omega_1 + \frac{k-3}{2}\omega_2] & \text{if }k\text{ odd}
  \end{cases}
\end{equation*}
In Theorem \ref{UNCproof} it is shown, that the ideal always has a generating set consisting of two elements, though there isn't given given any method to do so.

\section{Main result}

We first prove a technical lemma. 

\begin{lemma}\label{shiftingLemma}
  Let $A$ be a noetherian ring, let $B$ be a ring that is a finitely generated $A$-module and let $b \in B$. Consider $B$ as an $A[X]$-module by mapping $X \mapsto b$. Then there is an isomorphism of $A[X]$-modules
  \begin{equation*}
    \Ext_A^i(B,A) \simeq \Ext_{A[X]}^{i+1}(B,A[X])
  \end{equation*}
for all $i \geq 0$. 
\end{lemma}

\begin{proof}
Set $B[X] = A[X] \tensor_A B$ with trivial $A[X]$-action. Multiplication on $B[X]$ by $X-b$ fits in to a short exact sequence
\begin{equation*}
  0 \to B[X] \mathop{\to}\limits^{X-b} B[X] \to B \to 0.
\end{equation*}
Apply the left exact functor $\Hom_{A[X]}(-, A[X])$
\begin{equation*}
{\tiny
  \xymatrix{
    \dots \ar[r] & \Ext_{A[X]}^i(B[X],A[X]) \ar[r] \ar[d]^\simeq & \Ext_{A[X]}^i(B[X],A[X]) \ar[r] \ar[d]^\simeq & \Ext_{A[X]}^{i+1}(B,A[X]) \ar[r] & \dots \\
     & A[X] \tensor_A \Ext_A^i(B,A) \ar[r] & A[X] \tensor_A \Ext_A^i(B,A) \ar[r] & \Cok \ar[r] & 0
  }
}
\end{equation*}
where the vertical isomorphisms come from the fact that $A \to A[X]$ is flat. The lower homomorphism makes the diagram commutative so by naturality of the isomorphisms it is multiplication by $X-b$ identifying the cokernel with $\Ext_A^i(B,A)$. Diagram chasing gives us a map
\begin{equation*}
  \Ext_A^i(B,A) \to \Ext_{A[X]}^{i+1}(B,A[X])
\end{equation*}
which is an isomorphism by the Five Lemma. 
\end{proof}

\begin{theorem}\label{UNCproof}
  Let $I \subseteq R = \set Z[X,Y]$ be an ideal such that we have an isomorphism of $R$-modules
  \begin{equation}\label{fusionDuality}
    \Hom_{\set Z}(R/I, \set Z) \simeq R/I.
  \end{equation}
  Then $I$ is generated by an $R$-regular sequence of length 2.
\end{theorem}

\begin{proof}
Necessarily from the duality \eqref{fusionDuality} $R/I$ is a finitely generated $\set Z$-module, so Lemma \ref{shiftingLemma} applied to the ring $\set Z$, the $\set Z$-module $R/I$ and distinguished element $x = X+I \in R/I$ gives us
\begin{equation*}
  \Ext_{\set Z}^i(R/I, \set Z) \simeq \Ext_{\set Z[X]}^{i+1}(R/I,\set Z[X]).
\end{equation*}
Now $R/I$ is still finitely generated as an $\set Z[X]$-module so the lemma applied once more with the element $y = Y+I \in R/I$ gives us
\begin{equation*}
  \Ext_{\set Z[X]}^{i+1}(R/I,\set Z[X]) \simeq \Ext_{\set Z[X,Y]}^{i+2}(R/I,\set Z[X,Y]).
\end{equation*}
With $i=0$ and the assumption we get
\begin{equation*}
  \Ext_R^2(R/I,R) \simeq \Hom_{\set Z}(R/I,\set Z) \simeq R/I.
\end{equation*}

Pick a unit $e \in R/I$. With the identification $\Ext_R^2(R/I,R) \simeq \Ext_R^1(I,R)$ this element correspond to a nonsplit short exact sequence
\begin{equation}\label{ses}
  0 \to R \to M \to I \to 0.
\end{equation}
The goal is to show that $M \simeq R^2$ for then the image of two generators under the surjection in \eqref{ses} will generate $I$. 

We prove first that $\Ext_R^i(M,R) = 0$ for all $i \geq 1$. Consider the long exact sequence associated to \eqref{ses} 
\begin{equation*}
  \Hom_R(R,R) \stackrel{p}{\to} \Ext_R^1(I,R) \to \Ext_R^1(M,R) \to \Ext_R^1(R,R)=0.
\end{equation*}
By construction we have $p(\id) = e \in \Ext_R^1(I,R) \simeq R/I$. Choose an $f \in R$ with $e^{-1} = f+I$. Since $p$ is an $R$-homomorphism $p(f) = fe = (f+I)e = 1$ so $p$ is surjective and $\Ext_R^1(M,R) = 0$. Let now $i \geq 2$. Since $R/I$ is Gorenstein it is Cohen-Macaulay and localizing at a prime ideal $\ideal p \in R$ containing $I$ the Auslander-Buchsbaum formula gives us $\pd_R(R/I) = 2$. Therefore $0 = \Ext_R^{i+1}(R/I,R) \simeq \Ext_R^i(I,R) \simeq \Ext_R^i(M,R)$.

Then induction on the length of a projective resolution on a given module $N$ gives $\Ext_R^i(M,N) = 0$ for all $i \geq 1$, i.e. $M$ is projective.

Now \cite[Theorem 4]{Qui76} says that all projective modules over $\set Z[X,Y]$ are actually free so $M \simeq R^k$. Choose any prime ideal $\ideal p \subseteq R$ not containing $I$. Localizing \eqref{ses} at $\ideal p$ we get
\begin{equation*}
  0 \to R_{\ideal p} \to M_{\ideal p} \to R_{\ideal p} \to 0
\end{equation*}
showing that $k=2$. Here we used that $I \cap (R \setminus \ideal p) \neq \emptyset$ so $I_{\ideal p}$ contains a unit. 
\end{proof}

\begin{remark}
Much of the theory here generalizes to higher ranks. When we consider a general ideal $I \subseteq R = \set Z[X_1, \dots, X_r]$ such that $\Hom_{\set Z}(R/I, \set Z) \simeq R/I$ then Lemma \ref{shiftingLemma} applied $r$ times still gives us
\begin{multline*}
  R/I \simeq \Hom_{\set Z}(R/I, \set Z) \simeq \Ext_{\set Z[X_1]}^1(R/I, \set Z[X_1]) \simeq \\
 \dots \simeq \Ext_{\set Z[X_1, \dots, X_r]}^r(R/I, \set Z[X_1, \dots, X_r]).
\end{multline*}
Locally the Auslander-Buchsbaum formula still gives us $\pd_R(R/I) = r$ so $\Ext_R^i(R/I,R) = 0$ for $i > r$. We also have our tool to give us global information from local data: By \cite[Theorem 4]{Qui76} any projective $\set Z[X_1, \dots, X_r]$ is free. However only for $r=2$ is a non-trivial extension of $I$ by $R$ in $\Ext_R^{r-1}(I,R) \simeq R/I$ equivalent to a non-split short exact sequence which is the beginning of our construction in the proof. 

In \cite{Ser63} it was proved that a quotient of a regular local ring of codimension 2 is a complete intersection ring if and only if it is Gorenstein. After more than 50 years this result has not been generalized to higher codimensions suggesting that a generalization of Theorem \ref{UNCproof} to higher ranks is not possible without further assumptions on the ideal. 
\end{remark}

\bibliography{bibl}

\end{document}